\documentclass[preprint,12pt]{elsarticle}

\usepackage{graphicx}
\usepackage{amssymb}

\usepackage{xcolor}
\usepackage{epsf}
\usepackage{amsbsy,amsmath}
\usepackage{mathtools}
\usepackage{mathrsfs}
\usepackage{amsfonts}
\usepackage{amssymb}
\usepackage{enumitem}
\usepackage{eucal}
\usepackage{cases}
\usepackage{graphics,mathrsfs}
\usepackage{amsthm}
\usepackage{secdot}
\usepackage{esint}
\usepackage{hyperref}
\usepackage{varwidth}
\usepackage{tasks}
\usepackage{showlabels}
\usepackage{xcolor,soul}

\usepackage{geometry}
\newtheorem{theorem}{Theorem}[section]
\newtheorem{lemma}[theorem]{Lemma}

\theoremstyle{definition}
\newtheorem{definition}[theorem]{Definition}

\theoremstyle{remark}
\newtheorem{remark}[theorem]{Remark}
\numberwithin{equation}{section}

\numberwithin{equation}{section}

\usepackage{xcolor}

\begin{document}

 \begin{frontmatter}

	\title{Elliptic problem in an exterior domain driven by a singularity with a nonlocal Neumann condition}

\author{D. Choudhuri\fnref{label2}}
\ead{dc.iit12@gmail.com}
\author{K. Saoudi\corref{cor1}\fnref{label3}}
\ead{kasaoudi@gmail.com}
%
\cortext[cor1]{Corresponding author}
\fntext[label2]{Department of Mathematics, National Institute of Technology Rourkela, Rourkela - 769008, India}
\fntext[label3]{Basic and Applied Scientific Research Center, Imam Abdulrahman Bin Faisal University,  Saudi Arabia}


\begin{abstract}
We prove the existence of ground state solution to the following problem.   
\begin{align*}
(-\Delta)^{s}u+u&=\lambda|u|^{-\gamma-1}u+P(x)|u|^{p-1}u,~\text{in}~\mathbb{R}^N\setminus\Omega\\
N_su(x)&=0,~\text{in}~\Omega
\end{align*}
where $N\geq2$, $\lambda>0$, $0<s,\gamma<1$, $p\in(1,2_s^*-1)$ with $2_s^*=\frac{2N}{N-2s}$.
 Moreover, $\Omega\subset\mathbb{R}^N$ is a smooth bounded domain, $(-\Delta)^s$ denotes the $s$-fractional Laplacian and finally $N_s$ denotes the nonlocal operator that describes the Neumann boundary condition which is given as follows.
\begin{align*}
N_{s}u(x)&=C_{N,s}\int_{\mathbb{R}^N\setminus\Omega}\frac{u(x)-u(y)}{|x-y|^{N+2s}}dy,~x\in\Omega.
\end{align*}
We further establish the existence of infinitely many bounded solutions to the problem.
	\begin{flushleft}
	\end{flushleft}
\end{abstract}

\end{frontmatter}


\section{Introduction}
As mentioned in the Abstract, we will take up the following problem to study.
\begin{align}\label{main}\tag{{\bf P}}
\begin{split}
(-\Delta)^{s}u+u&=\lambda|u|^{-\gamma-1}u+P(x)|u|^{p-1}u,~\text{in}~\mathbb{R}^N\setminus\Omega\\
N_su(x)&=0,~\text{in}~\Omega.
\end{split}
\end{align}
The function $P$ is a continuous function such that $$(P_1):~~~~~~ P(x)\geq\tilde{P}>0~\text{in}~\mathbb{R}^N\setminus\Omega$$ 
and $$\underset{|x|\rightarrow\infty}{\lim}P(x)=\tilde{P}.$$
The ``nonlocal normal derivative'' $N_{s}$ was first introduced by Dipierro et al. \cite{dip1}. The authors in \cite{dip1} proved that as $s\rightarrow 1^-$, the classical Neumann boundary condition is recovered in the following sense. 
\begin{eqnarray}
\underset{s\rightarrow 1^-}{\lim}\int_{\mathbb{R}^N\setminus\Omega}vN_su&=&\int_{\partial\Omega}v\frac{\partial u}{\partial\nu}\nonumber
\end{eqnarray}
where $\nu$ is an outward drawn normal to the boundary $\partial\Omega$. Elliptic problems considered in a exterior domain is a rarity in the literature. However, when we traced through the literature pertaining to the exterior domain problem, we found a few seminal works. One of them is due to Benci and Cerami \cite{B_and_C} who considered the problem \eqref{main} with $s=1$, $V(x)\equiv 1$, $\lambda=0$, $Q(x)\equiv 1$ and with a zero Dirichlet boundary condition. The authors in \cite{B_and_C} showed that it does not have a ground state solution. Thanks to the article due to Esteban \cite{Est1} who proved that the same problem with Neumann condition has a ground state solution to the following problem.
\begin{align}\label{esteban}
\begin{split}
-\Delta u+u&=|u|^{p-1}u,~\text{in}~\mathbb{R}^N\setminus\Omega\\
\frac{\partial u}{\partial\nu}&=0,~\text{on}~\partial\Omega.
\end{split}
\end{align}
A noteworthy work is due to Cao \cite{cao1} who studied the existence of positive solution to \eqref{esteban} under the assumption that 
$$P(x)\geq \tilde{P}-Ce^{-a|x|}|x|^{-m}~~\text{as}~~|x|\rightarrow\infty$$ together with the condition in $(P_1)$. Here $a=\frac{2(p+1)}{p-1}$, $m>N-1$ and $C>0$.  Continuing in this article \cite{cao1}, Cao further proved the existence of sign changing solution (also know as a {\it nodal} solution) under the additional assumption that
$$P(x)\geq \tilde{P}+Ce^{-\frac{p|x|}{p+1}}|x|^{-m}~~\text{as}~~|x|\rightarrow\infty$$ together with the condition in $(P_1)$ with $0<m<\frac{N-1}{2}$. Alves et al. \cite{alv1} proved that the results found in \cite{cao1} also holds true for the $p$-Laplacian operator and for a larger class of nonlinearity. The problem with $N=2$ and a nonlinearity of critical growth has also been considered by Alves in \cite{alv2}.\\
Off-late, the fractional Laplacian operator gained a mileage as far as attention is concerned as it naturally arises in many different contexts, viz. optimization, thin obstacle problem, finance, crystal dislocation, conservation laws, limits of quantum mechanics, material science and water waves to name a few. Interested readers may refer to the articles [13,19,20,31,32] and the references therein. We drew motivation from the work due to of Alves \cite{alv3} to study the problem \eqref{main}. To the best of our knowledge, there is no article in the literature that addresses the problem \eqref{main} driven by a singularity in an exterior domain and a Neumann boundary condition. We will first prove the existence of a nonnegative ground state solution to \eqref{main}. Capitalising on this proof, we will further show that the problem has infinitely many bounded solutions for a finite range of $\lambda$. The main result concerning the existence of a ground state solution is as follows.
\begin{theorem}\label{main_thm}
Suppose $p\in(1,2_s^*-1)$ and $(P_1)$ holds, then \eqref{main} has positive ground state solution. Further, there exists $\lambda_0>0$ such that \eqref{main} has infinitely many bounded solutions whenever $\lambda\in (0,\lambda_0)$.
\end{theorem}
\section{Preliminaries}
\noindent The section introduces the readers to some well known function spaces besides considering the following limiting problem.
\begin{align}\label{limitmain}\tag{\mbox{$P_{\infty}$}}
\begin{split}
(-\Delta)^{s}u+u&=\lambda|u|^{-\gamma-1}u+\tilde{P}|u|^{p-1}u,~\text{in}~\mathbb{R}^N\\
u&\in H^s(\mathbb{R}^N).
\end{split}
\end{align}
The operator $(-\Delta)^s$ is defined as follows.
\begin{align}\label{s_lap}
\begin{split}
(-\Delta)^{s}u(x)&=C_{N,s}P.V.\int_{\mathbb{R}^N}\frac{u(x)-u(y)}{|x-y|^{N+2s}}dy
\end{split}
\end{align}
where $C_{N,s}=2^{2s-1}\pi^{-\frac{N}{2}}\frac{\Gamma(\frac{N+2s}{2})}{|\Gamma(-s)|}$. We will denote $H^s(\mathbb{R}^N)$ to be the fractional Sobolev space equipped with the norm
$$\|u\|=\left(\frac{1}{2}\int_{\mathbb{R}^N}\int_{\mathbb{R}^N}\frac{|u(x)-u(y)|^2}{|x-y|^{N+2s}}dydx+\int_{\mathbb{R}^N}|u|^2dx\right)^{\frac{1}{2}}.$$
Let $D\subset\mathbb{R}^N$ be a smooth domain. We now define the fractional Sobolev space pertaining to an exterior domain as follows.
$$H_D^s=\left\{u:\mathbb{R}^N\rightarrow\mathbb{R}~\text{measurable}:\frac{1}{2}\iint_{\mathbb{R}^{2N}\setminus(D^c)^2}\frac{|u(x)-u(y)|^2}{|x-y|^{N+2s}}dydx+\int_{D}|u|^2dx<\infty\right\}$$
where $D^c=\mathbb{R}^N\setminus D$.
This space is equipped with the norm
$$\|u\|_{s}=\left(\frac{1}{2}\iint_{\mathbb{R}^{2N}\setminus(D^c)^2}\frac{|u(x)-u(y)|^2}{|x-y|^{N+2s}}dydx+\int_{D}|u|^2dx\right)^{\frac{1}{2}}.$$
This space $H_{D}^s$ is a Hilbert space with an inner product $\langle\cdot,\cdot\rangle_{H_{D}^s}$ given by
$$\langle u,v\rangle_{H_{D}^s}=\frac{1}{2}\iint_{\mathbb{R}^{2N}\setminus(D^c)^2}\frac{(u(x)-u(y))(v(x)-v(y))}{|x-y|^{N+2s}}dydx+\int_{D}uvdx.$$
We now state a few embedding results pertaining to the space $H_D^s$ which can be found in \cite{18}.
\begin{lemma}\label{emb}
\begin{enumerate}
\item Let $H^s(D)$ be the classical fractional Sobolev space equipped with the norm
$$\|u\|_{H^s(D)}^2=\frac{1}{2}\int_{D}\int_{D}\frac{|u(x)-u(y)|^2}{|x-y|^{N+2s}}dydx+\int_{D}|u|^2dx.$$
Since $D \times D\subset \mathbb{R}^{2N}\setminus(D^c)^2$, then the embedding $H_D^s\hookrightarrow H^s(D)$ is continuous.
\item The embedding $H^s(\mathbb{R}^N)\hookrightarrow H_D^s$ is continuous.
\item Since $H^s(D)\hookrightarrow L^p(D)$ is continuously for every $p\in\left[2,\frac{2N}{N-2s}\right]$, by $(1)$ we have 
$$H_D^s\hookrightarrow L^p(D)~\text{for all}~p\in\left[2,\frac{2N}{N-2s}\right].$$
\item If $D$ is bounded, we have the compact embedding
$$H_D^s\hookrightarrow L^p(D)~\text{for all}~p\in\left[1,\frac{2N}{N-2s}\right).$$
\end{enumerate}
\end{lemma}
\begin{definition}[Palais Smale condition \cite{kesavan}]\label{PS}
Let $X$ be a Banach space and $J:X\rightarrow\mathbb{R}$ a $C^1$ functional. It is said to satisfy the {\it Palais-Smale} condition $(PS)$ if the following holds: whenever $(u_n)\subset X$ is such that $J(u_n)$ is bounded and $J'(u_n)\rightarrow 0$ in $X^*$, the dual space of $X$, then $(u_n)$ has a convergent subsequence. 
\end{definition}
\begin{definition}[Mountain pass theorem of  Ambrosetti and Rabinowitz \cite{kesavan}]\label{MP_thm}
Let $J:X\rightarrow\mathbb{R}$ be a $C^1$ functional satisfying $(PS)$. Let $u_0, u_2\in X$, $c\in\mathbb{R}$ and $R>0$ such that 
\begin{enumerate}
\item $\|u_1-u_0\|>R$
\item $J(u_0), J(u_1)<c\leq J(v)$, for all $v$ such that $\|v-u_0\|=R$.
\end{enumerate}
Then $J$ has a critical value $\tilde{c}\geq c$ defined by 
$$\tilde{c}=\underset{\delta\in\mathcal{P}}{\inf}~\underset{t\in[0,1]}{\max} \{J(\delta(t))\}$$
where $\mathcal{P}$ is the collection of all continuous paths $\delta:[0,1]\rightarrow X$ such that $\delta(0)=u_0$ and $\delta(1)=u_1$.
\end{definition}

\subsection{Cut-off functional}
We first define the functional $I:H_{\Omega^c}^s\rightarrow \mathbb{R}$ corresponding to the problem in \eqref{main} as follows.
\begin{align}\label{funct}
\begin{split}
I(u)&=\frac{1}{2}\left(\frac{1}{2}\iint_{\mathbb{R}^{2N}\setminus(\Omega^c)^2}\frac{|u(x)-u(y)|^2}{|x-y|^{N+2s}}dydx+\int_{\mathbb{R}^N\setminus\Omega}|u|^2dx\right)-\frac{\lambda}{1-\gamma}\int_{\mathbb{R}^N\setminus\Omega}|u|^{1-\gamma}dx-\frac{1}{p+1}\int_{\mathbb{R}^N\setminus\Omega}|u|^{p+1}dx.
\end{split}
\end{align}
Note that, the functional $I$ is not $C^1$ over $H_{\Omega^c}^s(\mathbb{R}^N)$ due to the presence of the singular term. Towards this, we will define a cut-off functional to overcome this problem. We now prove the following Lemma which will be used to construct the cut-off functional.
\begin{lemma}
	\label{existence_positive_soln}
	Let $0<\gamma<1$, $\lambda,\mu>0$. Then the following problem 
	\begin{eqnarray}\label{auxprob_appendix}
	(-\Delta)^su+u&=&\lambda u^{-\gamma},~\text{in}~\mathbb{R}^N\setminus\Omega\nonumber\\
	u&>&0,~\text{in}~\Omega\nonumber\\
	u&=&0,~\text{in}~\Omega
	\end{eqnarray}
	has a unique weak solution in $H_{\Omega}^s$. This solution is denoted by $\underline{u}_{\lambda}$, satisfies $\underline{u}_{\lambda}\geq \epsilon_{\lambda} v_0$ a.e. in $\Omega^c$, where $\epsilon_{\lambda}>0$ is a constant.
\end{lemma}
\begin{proof}
	We follow the proof in \cite{giaco_1}. First, we note that an energy functional on $H_{\Omega^c}^s$ formally corresponding to \eqref{auxprob_appendix} can be defined as follows.
	\begin{align}
	\label{ef_aux}
	E(u)=&\frac{1}{2}\left(\frac{1}{2}\iint_{\mathbb{R}^{2N}\setminus(\Omega^c)^2}\frac{|u(x)-u(y)|^2}{|x-y|^{N+2s}}dydx+\int_{\mathbb{R}^N\setminus\Omega}|u|^2dx\right)-\frac{\lambda}{1-\gamma}\int_{\mathbb{R}^N\setminus\Omega}|u|^{1-\gamma}dx
	\end{align}
	for $u\in H_{\Omega}^s$. By the Poincar\'{e} inequality, this functional is coercive and continuous on $H_{\Omega^c}^s$. It follows that $E$ possesses a global minimizer $u_0\in H_{\Omega^c}^s$. Clearly, $u_0\neq 0$ since $E(0)=0>E(\epsilon v_0)$ for sufficiently small $\epsilon$ and some $v_0>0$ in $\mathbb{R}^N\setminus\Omega$.\\
	Second, we have the decomposition $u=u^+-u^-$. Thus if $u_0$ is a global minimizer for $E$, then so is $|u_0|$, by $E(|u_0|)\leq E(u_0)$. Clearly enough, the equality holds iff $u_0^-=0$ a.e. in $\mathbb{R}^N\setminus\Omega$. In other words we need to have $u_0\geq 0$, i.e. $u_0\in H_{\Omega^c}^s$ where 
	$${H_{\Omega^c}^s}^+=\{u\in H_{\Omega}^s:u\geq 0~\text{a.e. in}~\mathbb{R}^N\setminus\Omega\}$$
	is the positive cone in $H_{\Omega^c}^s$.\\
	Third, we will show that $u_0\geq \epsilon v_0>0$ holds a.e. in $\mathbb{R}^N\setminus\Omega$ for small enough $\epsilon$. Observe that, 
	\begin{align}\label{neg_der}
	\begin{split}
	E'(tv_0)|_{t=\epsilon}=&\epsilon\left(\frac{1}{2}\iint_{\mathbb{R}^{2N}\setminus(\Omega^c)^2}\frac{|u(x)-u(y)|^2}{|x-y|^{N+2s}}dydx+\int_{\mathbb{R}^N\setminus\Omega}|u|^2dx\right)-\lambda \epsilon^{-\gamma}\int_{\mathbb{R}^N\setminus\Omega}|u|^{1-\gamma}dx<0
	\end{split}
	\end{align}
	whenever $0<\epsilon\leq \epsilon_{\lambda}$ for some sufficiently small $\epsilon_{\lambda}$. We now show that $u_0\geq \epsilon_{\lambda}v_0$. On the contrary, suppose $w=(\epsilon_{\lambda}v_0-u_0)^+$ does not vanish identically in $\mathbb{R}^N\setminus\Omega$. Denote $$(\mathbb{R}^N\setminus\Omega)^+=\{x\in\mathbb{R}^N\setminus\Omega:w(x)>0\}$$ to be the positive cone in $H_{\Omega^c}^s$.
	We consider the function $\zeta(t)=E(u_0+tw)$ of $t\geq 0$. This function is convex owing to its definition over the convex set ${H_{\Omega}^s}^+$. Further $\zeta'(t)=\langle E'(u_0+tw),w \rangle$ is nonnegative and nondecreasing for $t>0$. Consequently for $0<t<1$ we have 
	\begin{align}\label{ineq_appendix}
	\begin{split}
	0\leq \zeta'(1)-\zeta'(t)&=\langle E'(u_0+w)-E'(u_0+tw),w\rangle\\
	&=\int_{(\mathbb{R}^N\setminus\Omega)^+}E'(u_0+w)dx-\zeta'(t)\\
	&<0
	\end{split}
	\end{align}
	by inequality \eqref{neg_der} and $\zeta'(t)\geq 0$ with $\zeta'(t)$ being nondecreasing for every $t>0$, which is a contradiction. Therefore $w=0$ in $\mathbb{R}^N\setminus\Omega$ and hence $u_0\geq \epsilon_{\lambda}v_0$ a.e. in $\mathbb{R}^N\setminus\Omega$.\\
	Finally, the functional $E$ being strictly convex on ${H_{\Omega^c}^s}^+$, we conclude that $u_0$ is the only critical point of $E$ in ${H_{\Omega^c}^s}^+$. 
\end{proof}
\noindent We now define the following cut-off function which will be used to create the required cut-off functional.
\[   
\bar{f}(x,t) = 
\begin{cases}
\lambda|t|^{-\gamma-1}t+|t|^{p-1}t, &~\text{if}~|t|>\underline{u}_{\lambda}\\
\lambda\underline{u}_{\lambda}^{-\gamma}+\underline{u}_{\lambda}^{p},&~\text{if}~|t|\leq \underline{u}_{\lambda}
\end{cases}\]
where $\underline{u}_{\lambda}$ is a solution to \eqref{auxprob_appendix}.
Define 
\begin{align}\label{cut-off_omega_comp}
\bar{I}(u)&=\frac{1}{2}\left(\frac{1}{2}\iint_{\mathbb{R}^{2N}\setminus(\Omega)^2}\frac{|u(x)-u(y)|^2}{|x-y|^{N+2s}}dydx+\int_{\mathbb{R}^N\setminus\Omega}|u|^2dx\right)-\int_{\mathbb{R}^N\setminus\Omega}\bar{F}(x,u)dx
\end{align}
where $\bar{F}(x,t)=\int_{0}^t\bar{f}(x,s)ds$. This functional $\bar{I}$ thus defined is in $C^1(H_{\Omega^c}^s;\mathbb{R})$ and it is standard to show that
\begin{align}\label{der_cut-off_omega_comp}
\langle\bar{I}'(u),v\rangle&=\frac{1}{2}\left(\frac{1}{2}\iint_{\mathbb{R}^{2N}\setminus(\Omega)^2}\frac{(u(x)-u(y))(v(x)-v(y))}{|x-y|^{N+2s}}dydx+\int_{\mathbb{R}^N\setminus\Omega}uvdx\right)-\int_{\mathbb{R}^N\setminus\Omega}\bar{f}(x,u)vdx
\end{align}
for all $v\in H_{\Omega^c}^s$.\\
Similarly, we treat the functional corresponding to the problem defined in \eqref{limitmain}. The functional is defined as follows.
\begin{align}\label{funct_inf}
\begin{split}
I_{\infty}(u)&=\frac{1}{2}\left(\frac{1}{2}\iint_{\mathbb{R}^{2N}}\frac{|u(x)-u(y)|^2}{|x-y|^{N+2s}}dydx+\int_{\mathbb{R}^N}|u|^2dx\right)-\frac{\lambda}{1-\gamma}\int_{\mathbb{R}^N}|u|^{1-\gamma}dx-\frac{1}{p+1}\int_{\mathbb{R}^N}|u|^{p+1}dx.
\end{split}
\end{align}
A similar modification as done to the functional $I$ yields us the following.
\begin{align}\label{cut-off_omega_comp_inf}
\bar{I}_{\infty}(u)&=\frac{1}{2}\left(\frac{1}{2}\iint_{\mathbb{R}^{2N}}\frac{|u(x)-u(y)|^2}{|x-y|^{N+2s}}dydx+\int_{\mathbb{R}^N}|u|^2dx\right)-\int_{\mathbb{R}^N}\bar{G}(x,u)dx
\end{align}
Again, it is easy to see that $\bar{I}_{\infty}\in C^1(H^s(\mathbb{R}^N);\mathbb{R})$ and \begin{align}\label{der_cut-off_omega_comp_inf}
\langle\bar{I}_{\infty}'(u),v\rangle&=\frac{1}{2}\left(\frac{1}{2}\iint_{\mathbb{R}^{2N}}\frac{(u(x)-u(y))(v(x)-v(y))}{|x-y|^{N+2s}}dydx+\int_{\mathbb{R}^N}uvdx\right)-\int_{\mathbb{R}^N}\bar{g}(x,u)vdx
\end{align}
for all $v\in H^s(\mathbb{R}^N)$.\\
Here \[   
\bar{g}(x,t) = 
\begin{cases}
\lambda|t|^{-\gamma-1}t+|t|^{p-1}t, &~\text{if}~|t|>\underline{u}_{\lambda}^{\infty}\\
\lambda(\underline{u}_{\lambda}^{\infty})^{-\gamma}+(\underline{u}_{\lambda}^{\infty})^{p},&~\text{if}~|t|\leq \underline{u}_{\lambda}^{\infty}
\end{cases}\]
where $\bar{G}(x,t)=\int_{0}^t\bar{g}(x,s)ds$. Further, $\underline{u}_{\lambda}^{\infty}$ is a solution to the following problem. 
\begin{eqnarray}\label{auxprob_appendix_inf}
(-\Delta)^su+u&=&\lambda u^{-\gamma},~\text{in}~\mathbb{R}^N~~~~~~~~~~~~~~(P_{\infty})\nonumber\\
u&>&0,~\text{in}~\mathbb{R}^N.
\end{eqnarray}
Existence of a unique solution to \eqref{auxprob_appendix_inf} can be proved by following {\it verbatim} of the proof of Lemma \ref{existence_positive_soln}.
\begin{remark}\label{key_obs}
Instead of studying the problem \eqref{PS}, we will study the following problem
\begin{align}\label{main_mod}\tag{{\bf P}'}
\begin{split}
(-\Delta)^{s}u+u&=\bar{f}(x,u),~\text{in}~\mathbb{R}^N\setminus\Omega\\
N_su(x)&=0,~\text{in}~\Omega.
\end{split}
\end{align}
This is because a solution to \eqref{main_mod} is also a solution to \eqref{main}.
\end{remark}
\section{Proof of the main theorem}
\noindent This Section is devoted to the proof of the main Theorem \eqref{main_thm}. We begin by stating a {\it Lions type} Lemma that will play a crucial role in the proof of the main theorem. 
\begin{lemma}[Refer \cite{alv3}]\label{Lions_type}
Let $D\subset\mathbb{R}^N$ be an exterior domain with smooth bounded boundary and let $(u_n)\subset H_D^s$ be a bounded sequence such that 
\begin{align}
\underset{n\rightarrow\infty}{\lim}\underset{y\in\mathbb{R}^N}{\sup}\int_{U(y,T)}|u_n|^2dx=0,
\end{align}
for some $T>0$ and $U(y,T)=B(y,T)\cap D$ with $U(y,T)\neq\phi$. Then 
\begin{align}
\underset{n\rightarrow\infty}{\lim}\int_D|u_n|^pdx~\text{for all}~p\in(2,2_s^*).
\end{align}
\end{lemma}
\noindent The next lemma proves that the functional $\bar{I}$ satisfies the Mountain pass geometry.
\begin{lemma}\label{MP_geometry}
The functional $\bar{I}$ verifies the mountain pass geometry for $\lambda\in(0,\lambda_0)$ with $\lambda_0<\infty$.
\end{lemma}
\begin{proof}
Since $p\in\left(1,\frac{N+2s}{N-2s}\right)$ and $P$ is bounded, hence, by Sobolev embedding we obtain$$\bar{I}(u)\geq\frac{1}{2}\|u\|_{H_{\Omega^c}^s}^2-\frac{C_1\|P\|_{\infty}}{p+1}\|u\|_{H_{\Omega^c}^s}^{p+1}-\frac{\lambda C_2}{1-\gamma}\|u\|_{H_{\Omega^c}^s}^{1-\gamma}$$
where $C_1, C_2>0$ are uniform constants that are independent of the choice of $u$. Now for a small $\lambda>0$, say $\lambda_0$, we have that $\frac{1}{2}\|u\|_{H_{\Omega^c}^s}^2-\frac{\lambda C_2}{1-\gamma}\|u\|_{H_{\Omega^c}^s}^{1-\gamma}>0$. Note that, this positivity holds for any $\lambda\in(0,\lambda_0)$. For a sufficiently small $\|u\|_{H_{\Omega^c}^s}=r$, we further have $a(r)=\frac{1}{2}r^2-\frac{C_1\|P\|_{\infty}}{p+1}r^{p+1}-\frac{\lambda C_2}{1-\gamma}r^{1-\gamma}>0$. Therefore, to sum it up we have a pair $(\lambda,r)$ such that 
$$\bar{I}(u)\geq a(r)>0$$
for any $\lambda\in(0,\lambda_0)$ and for every $u$ such that $\|u\|_{H_{\Omega^c}^s}=r$. On the other hand, taking $u\in H_{\Omega^c}^s\setminus\{0\}$ and $t\geq 0$ we have 
$$\bar{I}(tu)=\frac{t^2}{2}\|u\|_{H_{\Omega^c}^s}^2-\frac{\lambda t^{1-\gamma}}{1-\gamma}\int_{\Omega^c}|u|^{1-\gamma}dx-\frac{t^{p+1}}{p+1}\int_{\Omega^c}P(x)|u|^{p+1}.$$
Since $p+1>2>1-\gamma$, we have $\bar{I}(tu)\rightarrow-\infty$ as $t\rightarrow\infty$. This verifies the second condition of the Mountain pass theorem.
\end{proof}
\begin{remark}\label{obs_pos} We now perform an apriori analysis on a solution (if it exists). Suppose $u$ is a solution to \eqref{main}, then we observe the following
\begin{enumerate}
	\item $I(u)=I(|u|)$. This implies that $u^-=0$ a.e. in $\Omega^c$.
	\item In fact $u>0$ a.e. in $\Omega^c$ due to the presence of the singular term.
	\end{enumerate}
Thus without loss of generality, we assume that the solution is positive. 
\end{remark}
Precisely, we now have the following result.
\begin{lemma}[Apriori analysis]\label{u_greater_u_lambda}
	Fix a $\lambda\in(0,\lambda_0)$. Then a solution of \eqref{main}, say $u>0$, is such that $u\geq \underline{u}_{\lambda}$ a.e. in $\mathbb{R}^N\setminus\Omega$.
\end{lemma}
\begin{proof}
	Fix $\lambda\in(0,\lambda_0)$ and let $u\in H_{\Omega^c}^s$ be a positive solution to \eqref{main} and $\underline{u}_{\lambda}>0$ be a solution to \eqref{auxprob_appendix}. We will show that $u\geq \underline{u}_{\lambda}$ a.e. in $\mathbb{R}^N\setminus\Omega$. Indeed, let $\underline{\Omega}=\{x\in\mathbb{R}^N\setminus\Omega:u(x)<\underline{u}_{\lambda}(x)\}$ and from the equation satisfied by $u$, $\underline{u}_{\lambda}$, we have 
	\begin{align}\label{comp_1}
	\langle(-\Delta)^s\underline{u}_{\lambda}-(-\Delta)^su,\underline{u}_{\lambda}-u\rangle_{\underline{\Omega}}+\int_{\underline{\Omega}}|\underline{u}_{\lambda}-u|^2dx&\leq\lambda\int_{\underline{\Omega}}(\underline{u}_{\lambda}^{-\gamma}-u^{-\gamma})(\underline{u}_{\lambda}-u)dx\leq 0.
	\end{align}
	Further, by the Simon's inequality we have 
	\begin{align}\label{comp_2}
	\langle(-\Delta)^s\underline{u}_{\lambda}-(-\Delta)^su,\underline{u}_{\lambda}-u\rangle_{\underline{\Omega}}&\geq 0. 
	\end{align}
	Hence, from \eqref{comp_1} and \eqref{comp_2}, we obtain $u\geq\underline{u}_{\lambda}$ a.e. in $\Omega^c$.
\end{proof}
The Lemma \eqref{MP_geometry} allows us to apply the Mountain pass theorem without the Palais-Smale condition (Definition \ref{PS}) to find a sequence $(u_n)\subset H_{\Omega^c}^s$ such that 
\begin{align}\label{ps_2}\bar{I}(u_n)\rightarrow c_1~\text{and}~\bar{I}'(u_n)\rightarrow 0\end{align}
where
\begin{align}\label{ps_3}c_1&=\underset{u\in H_{\Omega^c}^s\setminus\{0\}}{\inf}\underset{t\geq 0}{\sup}\bar{I}(tu).\end{align}
Moreover, we further have 
\begin{align}\label{ps_4}c_1&=\underset{u\in\mathcal{N}}{\inf}\bar{I}(u)\end{align}
where 
\begin{align}\label{ps_5}\mathcal{N}&=\{u\in H_{\Omega^c}^s\setminus\{0\}:\langle\bar{I}'(u),u\rangle=0\}\end{align}
is called a Nehari manifold. Henceforth, we say that $u\in H_{\Omega^c}^s$ is a ground state solution to \eqref{main_mod} when 
\begin{align}\label{ps_1}\bar{I}(u)=c_1~\text{and}~\bar{I}'(u)=0.\end{align}
We further recall that by a {\it ground state} solution we mean that a function $\tilde{u}$
\begin{lemma}\label{energy_comp}
Suppose $(P_1)$ holds, then $$0<c_1<c_{\infty}$$
whenever $\lambda\in(0,\lambda_0)$.
\end{lemma}
\begin{proof}
Let $\overline{u}$ be a nontrivial ground state solution of $(P_{\infty})$ and define $u_n(x)=\overline{u}(x-\alpha_n)$ where $\alpha_n=(n,0,\cdots,0)\in\mathbb{R}^N$. From \eqref{ps_2}, 
\begin{align}\label{cond_1}
c_1&=\underset{t\geq 0}{\max}\{\bar{I}(tu)\}.
\end{align}
We have for every $t>0$ consider the function
\begin{align}\label{cond_2}
f(t)&=\frac{t^2}{2}\|u_n\|_{H_{\Omega^c}^s}^2-\frac{\lambda t^{1-\gamma}}{1-\gamma}\int_{\Omega^c}|u_n|^{1-\gamma}dx-\frac{t^{p+1}}{p+1}\int_{\Omega^c}P(x)|u_n|^{p+1}dx.
\end{align}
It is clear that $f(0)=0$, $f(t)>0$ for $t$ small enough and $f(t)<0$ for $t$ large enough. Therefore, there exists a unique $\gamma_n\in(0,\infty)$ such that
\begin{align}\label{cond_3}
f(\gamma_n)&=\bar{I}(\gamma_n u_n)=\underset{t\geq 0}{\max}\{\bar{I}(tu_n)\}.
\end{align}
Therefore, $f'(\gamma_n)=0$ which amounts to saying that
\begin{align}\label{cond_4}
\frac{1}{2}\iint_{\mathbb{R}^{2N}\setminus\Omega^2}\frac{|u_n(x)-u_n(y)|^2}{|x-y|^{N+2s}}dydx+\int_{\Omega^c}|u_n|^2dx&=\lambda \gamma_n^{-1-\gamma}\int_{\Omega^c}|u_n|^{1-\gamma}dx+\gamma_n^{p-1}\int_{\Omega^c}P(x)|u_n|^{p+1}dx.
\end{align}
From the definition of $c_1$ given in \eqref{ps_4} we obtain
\begin{align}\label{cond_5}
\begin{split}
c_1\leq&\bar{I}(\gamma_n u_n)\\
=&\bar{I}_{\infty}(u)-\frac{\gamma_n^2}{2}\left(\frac{1}{2}\iint_{\Omega\times\Omega}\frac{|u_n(x)-u_n(y)|^2}{|x-y|^{N+2s}}dydx+\int_{\Omega}|u_n|^2dx\right)\\
&+\frac{\gamma_n^{p+1}}{p+1}\int_{\mathbb{R}^{N}\setminus\Omega}(\tilde{P}-P(x))|u_n|^{p+1}dx+\frac{\gamma_n^{p+1}}{p+1}\int_{\Omega}\tilde{P}|u_n|^{p+1}dx+\frac{\lambda\gamma_n^{1-\gamma}}{1-\gamma}\int_{\Omega}|u_n|^{1-\gamma}dx\\
=&\bar{I}_{\infty}(\gamma_n u_n)-\frac{a_n\gamma_n^2}{2}+\frac{\gamma_n^{p+1}}{p+1}\int_{\Omega}\tilde{P}|u_n|^{p+1}dx+\frac{\gamma_n^{p+1}}{p+1}\int_{\mathbb{R}^{N}\setminus\Omega}(\tilde{P}-P(x))|u_n|^{p+1}dx\\
&+\frac{\lambda\gamma_n^{1-\gamma}}{1-\gamma}\int_{\Omega}|u_n|^{1-\gamma}dx
\end{split}
\end{align}
where 
$$a_n=\frac{1}{2}\iint_{\Omega\times\Omega}\frac{|u_n(x)-u_n(y)|^2}{|x-y|^{N+2s}}dydx+\int_{\Omega}|u_n|^2dx.$$
We found from \eqref{cond_4} that $(\gamma_n)$ is bounded. For if not, then there exists a subsequence of $(\gamma_n)$, still denoted as $(\gamma_n)$, such that $\gamma_n\rightarrow\infty$. Moreover, as $n\rightarrow\infty$, we have 
\begin{align}\label{cond_6}\frac{1}{2}\iint_{\mathbb{R}^{2N}\setminus\Omega^2}\frac{|u_n(x)-u_n(y)|^2}{|x-y|^{N+2s}}dydx+\int_{\Omega^c}|u_n|^2dx&\rightarrow \|\bar{u}\|^2.\end{align}
Employing this in \eqref{cond_4} leads to an absurdity that $\|\bar{u}\|^2=\infty$. Therefore, $(\gamma_n)$ is bounded.Thus, up to a subsequence we have $\gamma_n\rightarrow\gamma_0$. We now claim that $\gamma_0=1$.\\
On making a change of variables $\tilde{x}=x-\alpha_n$, $\tilde{y}=y-\alpha_n$, we see that 
\begin{align}\label{cond_7}
\begin{split}
\|u_n\|_{H_{\Omega^c}^s}^2&=\frac{1}{2}\iint_{\mathbb{R}^{2N}}\chi_{\mathbb{R}^{2N}\setminus\Omega^2}(x,y)\frac{|\bar{u}(x-\alpha_n)-\bar{u}(y-\alpha_n)|^2}{|x-y|^{N+2s}}dydx+\int_{\mathbb{R}^N}\chi_{\mathbb{R}^N\setminus\Omega}|\bar{u}(x-\alpha_n)|^2dx\\
&=\frac{1}{2}\iint_{\mathbb{R}^{2N}}\chi_{\mathbb{R}^{2N}\setminus\Omega^2}(x+\alpha_n,y+\alpha_n)\frac{|\bar{u}(x)-\bar{u}(y)|^2}{|x-y|^{N+2s}}dydx+\int_{\mathbb{R}^N}\chi_{\mathbb{R}^N\setminus\Omega}(x+\alpha_n)|\bar{u}(x)|^2dx.
\end{split}
\end{align}
Since $|\alpha_n|\rightarrow\infty$,
$$\chi_{\mathbb{R}^{2N}\setminus\Omega^2}(x+\alpha_n,y+\alpha_n)\frac{|\bar{u}(x)-\bar{u}(y)|^2}{|x-y|^{N+2s}}\rightarrow \frac{|\bar{u}(x)-\bar{u}(y)|^2}{|x-y|^{N+2s}}~\text{a.e.}~(x,y)\in\mathbb{R}^{N}\times\mathbb{R}^N$$
and 
$$\chi_{\mathbb{R}^N\setminus\Omega}(x+\alpha_n)|\bar{u}(x)|^2\rightarrow|\bar{u}(x)|^2~\text{a.e.}~x\in\mathbb{R}^N.$$
Furthermore, 
$$\left|\chi_{\mathbb{R}^{2N}\setminus\Omega^2}(x+\alpha_n,y+\alpha_n)\frac{|\bar{u}(x)-\bar{u}(y)|^2}{|x-y|^{N+2s}}\right|\leq \frac{|\bar{u}(x)-\bar{u}(y)|^2}{|x-y|^{N+2s}}\in L^1(\mathbb{R}^{N}\times\mathbb{R}^N)$$
and
$$|\chi_{\mathbb{R}^N\setminus\Omega}(x+\alpha_n)|\bar{u}(x)|^2|\leq\bar{u}(x)|^2\in L^1(\mathbb{R}^N).$$
Thus, by the Lebesgue's dominated convergence theorem we have 
\begin{align}\label{cond_8}
\|u_n\|_{H_{\Omega^c}^s}^2\rightarrow\|\bar{u}\|~\text{as}~n\rightarrow\infty
\end{align}
and thus 
\begin{align}\label{cond_8'}
\gamma_n^{1+\gamma}\|u_n\|_{H_{\Omega^c}^s}^2\rightarrow\gamma_0^{1+\gamma}\|\bar{u}\|~\text{as}~n\rightarrow\infty.
\end{align}
By condition $(P_1)$ we have,
$$\gamma_n^{p+\gamma}\chi_{\mathbb{R}^N\setminus\Omega}(x+\alpha_n)P(x+\alpha_n)|\bar{u}(x)|^{p+1}\rightarrow\gamma_0^{p+\gamma}|\tilde{u}(x)|^{p+1}~\text{a.e.}~x\in\mathbb{R}^N$$
and by the boundedness of $P$ and $(\gamma_n)$ we obtain
\begin{align}\label{cond_9}
|\gamma_n^{p+\gamma}\chi_{\mathbb{R}^N\setminus\Omega}(x+\alpha_n)P(x+\alpha_n)|\bar{u}(x)|^{p+1}|\leq\gamma_0^{p+\gamma}M|\tilde{u}(x)|^{p+1}~\text{a.e.}~x\in\mathbb{R}^N\in L^1(\mathbb{R}^N).
\end{align}
Thus by the Lebesgue's dominated convergence we have 
\begin{align}\label{cond_10}
\gamma_n^{p+\gamma}\int_{\mathbb{R}^N\setminus\Omega}P(x)|u_n|^{p+1}dx\rightarrow\gamma_0^{p+\gamma}\int_{\mathbb{R}^N}\tilde{P}|\bar{u}(x)|^{p+1}dx.
\end{align}
Similarly, on the other hand,
$$\lambda \chi_{\mathbb{R}^N\setminus\Omega}(x+\alpha_n)|\bar{u}(x)|^{1-\gamma}\rightarrow\lambda|\bar{u}(x)|^{1-\gamma}~\text{a.e.}~x\in\mathbb{R}^N$$
and
$$\int_{\mathbb{R}^N}\lambda \chi_{\mathbb{R}^N\setminus\Omega}(x+\alpha_n)|\bar{u}(x)|^{1-\gamma}\leq\int_{\mathbb{R}^N}\lambda|\bar{u}(x)|^{1-\gamma}~\text{a.e.}~x\in\mathbb{R}^N.$$
Needless to say, by the Lebesgue's dominated convergence theorem again we have
\begin{align}\label{cond_11}
\int_{\mathbb{R}^N\setminus\Omega}|u_n|^{1-\gamma}dx\rightarrow\int_{\mathbb{R}^N}|\bar{u}(x)|^{1-\gamma}dx.
\end{align}
Since, $\bar{u}$ is a solution to $(P_{\infty})$ and together with the limits \eqref{cond_8}, \eqref{cond_10} and \eqref{cond_11} gives $\gamma_0=1$. Further, by \eqref{cond_5} we also have
\begin{align}\label{cond_12}
c_1&\leq \bar{I}_{\infty}(\bar{u})-\frac{t_n\gamma_n^2}{2}+s_n=c_{\infty}-\frac{t_n\gamma_n^2}{2}+s_n.
\end{align}
Here, $$b_n=\frac{\gamma_n^{p+1}}{p+1}\left(\int_{\Omega}\tilde{P}|u_n|^{p+1}dx+\int_{\mathbb{R}^{N}\setminus\Omega}(\tilde{P}-P(x))|u_n|^{p+1}dx\right)+\frac{\lambda\gamma_n^{1-\gamma}}{1-\gamma}\int_{\Omega}|u_n|^{1-\gamma}dx.$$
 We will show that $a_n\rightarrow 0$ and $b_n\rightarrow\frac{\gamma_0^{p+1}}{p+1}\int_{\mathbb{R}^{N}}(\tilde{P}-P(x))|\bar{u}|^{p+1}dx<0$ as $n\rightarrow\infty$ which is sufficient to show that $c_1<c_{\infty}$.\\
 In order to verify this, we first note that $\bar{u}\in H^s(\mathbb{R}^N)$ and this gives $a_n\rightarrow 0$ as $n\rightarrow\infty$. Using similar argument, we get $b_n\rightarrow\frac{\gamma_0^{p+1}}{p+1}\int_{\mathbb{R}^{N}}(\tilde{P}-P(x))|\bar{u}|^{p+1}dx<0$.
 Therefore, from \eqref{cond_12} we have $c_1<c_{\infty}$.
\end{proof}
\begin{proof}[{\bf Proof of Theorem \ref{main_thm}}]
From \eqref{ps_2} there exists a sequence $(u_n)\subset H_{\Omega^c}^s$ such that 
$$\bar{I}(u_n)\rightarrow c_1~\text{and}~\bar{I}'(u_n)\rightarrow 0~\text{as}~n\rightarrow\infty.$$
Now since $(u_n)$ is bounded, there exists a subsequence still denoted by $(u_n)$, such that $u_n\rightharpoonup u$ for some $u\in H_{\Omega^c}^s$ and $\bar{I}'(u)=0$. The condition that $\bar{I}'(u_n)\rightarrow 0$ implies that each $u_n$ cannot be zero over a non-zero subset of $\Omega^c$. For if it does, then it leads to a contradiction that $\bar{I}(u_n)$ is finite. We claim that $u\neq 0$. Let us assume on the contrary that $u=0$. Since $c_1>0$, the Lemma \ref{Lions_type} guarantees the existence of $\rho,\beta>0$ and $(z_n)\subset\Omega^c$ with $|z_n|\rightarrow\infty$ as $n\rightarrow\infty$ such that 
$$\int_{B_r(z_n)\cap\Omega^c}|u_n|^2dx\geq\beta~\forall n\in\mathbb{N}.$$
Then for each fixed $T>0$, there exists $n_0=n_0(T)\in\mathbb{N}$ such that
$$B(0,T)\subset\mathbb{R}^N\setminus(\Omega-z_n),~\forall n\geq n_0.$$
Let $w_n(x)=u_n(x+z_n)$ for $x\in\Omega$ and as $w_n(x)=u_n(x)$ for $x\in \mathbb{R}^N\setminus\Omega$. This defines $w_n$ in the entire $\mathbb{R}^N$. Then we have for some subsequence of $(w_n)$, still denoted as $(w_n)$, which is bounded in $H^s(B(0,T))$ for all $T>0$. This is because $(u_n)$ is bounded in $H_{\Omega^c}^s$ and therefore there exists a positive constant $C$ such that 
\begin{align}\label{cond_13}
\begin{split}
C&\geq\iint_{\mathbb{R}^{2N}\setminus\Omega^2}\frac{|u_n(x)-u_n(y)|^2}{|x-y|^{N+2s}}dydx+\int_{\mathbb{R}^N\setminus\Omega}|u_n|^2dx\\
&=\iint_{\mathbb{R}^{2N}\setminus(\Omega-z_n)^2}\frac{|w_n(x)-w_n(y)|^2}{|x-y|^{N+2s}}dydx+\int_{\mathbb{R}^N\setminus(\Omega-z_n)}|w_n|^2dx\\
&\geq\iint_{B(0,T)\times B(0,T)}\frac{|w_n(x)-w_n(y)|^2}{|x-y|^{N+2s}}dydx+\int_{B(0,T)}|w_n|^2dx=\|w_n\|_{H^s(B(0,T))}^2.
\end{split}
\end{align}
This suggests that there exists a subsequence of $(w_n)$, still denoted as $(w_n)$, and a $v\in H_{\text{loc}}^s(\mathbb{R}^N)$ such that 
$w_n\rightharpoonup w$ in $H^s(B(0,T))$ as $n\rightarrow\infty$.\\
Further, by the lower semicontinuity of the norm we have 
$$\|w\|_{H^s(B(0,T))}\leq \underset{n\rightarrow\infty}{\liminf}\|w_n\|_{H^s(B(0,T))}\leq C$$ for every $T>0$. It follows from this that $w\in H^s(\mathbb{R}^N)$.\\
Let $\varphi\in H^s(\mathbb{R}^N)$ be a test function with bounded support. Since, $\bar{I}'(u_n)=o_n(1)$,
\begin{align}\label{cond_14}
\langle\bar{I}'(u_n),\varphi(\cdot-z_n)\rangle&=o_n(1).
\end{align}
Hence, 
\begin{align}\label{cond_15}
\begin{split}
\frac{1}{2}\iint_{\mathbb{R}^N\setminus(\Omega-z_n)^2}\frac{(w_n(x)-w_n(y))(\varphi(x)-\varphi(y))}{|x-y|^{N+2s}}dydx+\int_{\mathbb{R}^N\setminus(\Omega-z_n)}w_n\varphi dx\\=\int_{\mathbb{R}^N\setminus(\Omega-z_n)}P(x+z_n)|w_n|^{p-1}w_n\varphi dx+\lambda\int_{\mathbb{R}^N\setminus(\Omega-z_n)}|w_n|^{-\gamma-1}w_n\varphi dx.
\end{split}
\end{align}
By the weak convergence of $w_n$ to $w$ in $H^s(B(0,T))$, we realize that
\begin{align}\label{cond_15'}
\begin{split}
\frac{1}{2}\iint_{\mathbb{R}^{2N}\setminus(\Omega-z_n)^2}\frac{(w_n(x)-w_n(y))(\varphi(x)-\varphi(y))}{|x-y|^{N+2s}}dydx+\int_{\mathbb{R}^N\setminus(\Omega-z_n)}w_n\varphi dx\\
\rightarrow \frac{1}{2}\iint_{\mathbb{R}^{2N}}\frac{(w(x)-w(y))(\varphi(x)-\varphi(y))}{|x-y|^{N+2s}}dydx+\int_{\mathbb{R}^N}w\varphi dx~\text{as}~n\rightarrow\infty.
\end{split}
\end{align}
Now suppose $|z_n|\rightarrow\infty$, then we have
$$P(x+z_n)\rightarrow\tilde{P}~\text{a.e.}~x\in\mathbb{R}^N,~\text{as}~n\rightarrow\infty$$
and therefore 
$$P(x+z_n)|w_n|^{p-1}w_n\varphi\rightarrow\tilde{P}|w|^{p-1}w\varphi~\text{a.e.}~x\in\mathbb{R}^N,~\text{as}~n\rightarrow\infty.$$
Further,
$$|w_n|^{-\gamma-1}w_n\varphi\rightarrow|w|^{-\gamma-1}w\varphi~\text{a.e.}~x\in\mathbb{R}^N,~\text{as}~n\rightarrow\infty.$$
These limits in combination with the boundedness $(w_n)$ in $L^{p+1}(\mathbb{R}^N\setminus\Omega)$ permits us to apply \cite{27}, 4.6 to obtain
\begin{align}\label{cond_16}
\int_{B(0,T)}P(x+z_n)|w_n|^{p-1}w_n\varphi dx\rightarrow\int_{B(0,T)}\tilde{P}|w|^{p-1}w\varphi dx.
\end{align}
Note that since $w_n\rightharpoonup w$ as $n\rightarrow\infty$, we have by the compact embedding given in Lemma \ref{emb}; $(4)$ that $w_n\rightarrow w$ in $L^{p+1}(\mathbb{R}^N)$. Further, since $\varphi$ is bounded and with a bounded support, hence 
$$|w_n|^{p-\gamma-1}w_n\varphi \rightarrow|w|^{-\gamma-1}w\varphi dx$$
and therefore
\begin{align}\label{cond_17}
\int_{B(0,T)}|w_n|^{p-\gamma-1}w_n\varphi dx\rightarrow\int_{B(0,T)}|w|^{-\gamma-1}w\varphi dx.
\end{align}
Note that, this also implies that $w$ cannot be zero over a subset (of $\mathbb{R}^N$) of non-zero measure (refer Appendix). Thus, \eqref{cond_15}-\eqref{cond_17} gives
$$\langle\bar{I}_{\infty}'(w),\varphi\rangle=0.$$
Now by density we extend our test function space to $H^s(\mathbb{R}^N)$ and the last equality gives that $w$ is a nontrivial solution of $(P_{\infty})$. On computing the following
\begin{align}\label{cond_18}
\begin{split}
c_{\infty}&\leq\bar{I}_{\infty}(w)-\frac{1}{2}\langle\bar{I}_{\infty}'(w),w\rangle\\
&=\left(\frac{1}{2}-\frac{1}{p+1}\right)\int_{\mathbb{R}^N}\tilde{P}|w|^{p+1}dx+\lambda\left(\frac{1}{2}-\frac{1}{1-\gamma}\right)\int_{\mathbb{R}^N}|w|^{1-\gamma}dx\\
&\leq\left(\frac{1}{2}-\frac{1}{p+1}\right)\int_{\mathbb{R}^N}\tilde{P}|w|^{p+1}dx\\
&\leq\underset{n\rightarrow\infty}{\liminf}\left(\frac{1}{2}-\frac{1}{p+1}\right)\int_{\mathbb{R}^N\setminus(\Omega-z_n)}P(x+z_n)|w_n|^{p+1}dx\\
&=\underset{n\rightarrow\infty}{\liminf}\left(\frac{1}{2}-\frac{1}{p+1}\right)\int_{\mathbb{R}^N\setminus\Omega}P(x)|u_n|^{p+1}dx\\
&=\underset{n\rightarrow\infty}{\liminf}(\bar{I}(u_n)-\frac{1}{2}\langle\bar{I}'(u_n),u_n\rangle)=c_1
\end{split}
\end{align}
we get a violation of the Lemma \ref{Lions_type}. Thus $u\neq 0$. As seen earlier, we have that $\bar{I}(|u|)=\bar{I}(u)$ which yields $u\geq 0$ a.e. in $\Omega^c$. However, by the Appendix we conclude that $u>0$ a.e. in $\Omega^c$. This proves the existence of a positive ground state solution to \eqref{main}.
\end{proof}
\noindent{\it Existence of infinitely many solutions}:~
\noindent We now give the symmetric mountain pass theorem \cite{col}.
\begin{theorem}\label{symmMPT}
	(Symmetric mountain pass theorem) Let $X$ be an infinite dimensional Banach space. $Y$ is a finite dimensional Banach space and $X=Y \bigoplus Z$. For any $c>0$ if $I\in C^1(X,\mathbb{R})$ satisfies $(Ce)_c$ and 
	\begin{enumerate}
		\item $I$ is even and $I(0)=0$ for all $u\in X$ 
		\item There exists $r>0$ such that $I(u)\geq R$ for all $u\in B_r(0)=\{u\in X:\|u\|_X\leq r\}$
		\item For any finite dimensional subspace $\bar{X}\subset X$, there exists $r_0=r(\bar{X})>0$ such that $I(u)\leq 0$ on $\bar{X}\setminus B_{r_0}(0_{\bar{X}})$, where $0_{\bar{X}}$ is the null vector in $\bar{X}$
	\end{enumerate}
	then there exists an unbounded sequence of critical values of $I$ characterized by a minimax argument.
\end{theorem}
\noindent{\it $(PS)$ condition is satisfied by $\bar{I}$}:~Suppose $(u_n)$ is a sequence such that $\bar{I}(u_n)\rightarrow c$ and $\bar{I}'(u_n)\rightarrow 0$ as $n\rightarrow\infty$. This implies that $(u_n)$ is bounded in $H_{\Omega^c}^s$. Therefore there exists a subsequence, still denoted by $(u_n)$, such that 
\begin{align}\label{conv1}
\begin{split}
& u_n\rightharpoonup u~\text{in}~H_{\Omega^c}^s,\\
& u_n\rightarrow u~\text{in}~L^r(\Omega^c), 1\leq r< 2_s^*.
\end{split}
\end{align}   
Consider the following. 
\begin{align}\label{cond_19}
\begin{split}
o(1)&=\langle\bar{I}'(u_n),u_n-u\rangle\\
&=\langle u_n,u_n-u\rangle-\int_{\Omega^c}P(x)|u_n|^{p-1}u_n(u_n-u)dx-\lambda\int_{\Omega}|u_n|^{-\gamma-1}u_n(u_n-u)dx\\
&=\|u_n\|^2-\|u\|^2+o(1)
\end{split}
\end{align}
where we have also used the outcome of the Appendix. Therefore $u_n\rightarrow u$ in $H_{\Omega^c}^s$.\\
Let us develope some prerequisites. It is well known that if $X$ is a Banach space then we have that 
$$X=\underset{i\geq 1}{\bigoplus}X_i$$
where $X_{i}=\text{span}\{e_j\}_{j\geq i}$. Define
$$Y_m=\underset{1\leq j \leq m}{\bigoplus}X_j$$
$$Z_m=\underset{j\geq m}{\bigoplus}X_j.$$
Clearly, $Y_m$ is a finite dimensional subspace of $X$, for each $m$. Let $X=H_{\Omega^c}^s$.
By the equivalence of norm in $Y_m$, we have
\begin{align}\label{conv24}
\begin{split}
\bar{I}(u)&=\frac{1}{2}\|u\|_{H_{\Omega^c}^s}^2-\frac{\lambda }{1-\gamma}\int_{H_{\Omega^c}^s}|u|^{1-\gamma}dx-\frac{1}{p+1}\int_{H_{\Omega^c}^s}P(x)|u|^{p+1}dx\\
&\leq\frac{1}{2}\|u\|_{H_{\Omega^c}^s}^2-\frac{\lambda C_3}{1-\gamma}\|u\|_{H_{\Omega^c}^s}^{1-\gamma}-\frac{\tilde{P}C_4}{p+1}\|u\|_{H_{\Omega^c}^s}^{p+1}\leq 0
\end{split}
\end{align}
Thus for any finite dimensional subspace $\bar{X}\subset H_{\Omega^c}^s$, there exists a sufficiently large $r_0=r(\bar{X})$ for which we have $\bar{I}(u)\leq 0$ whenever $\|u\|\geq r_0$. Hence, by the Theorem \ref{symmMPT}, there exists an unbounded sequence of critical values of $\bar{I}$ characterized by a minimax argument. In other words, from the Remark \ref{key_obs} the problem in \eqref{main_mod} has infinitely many solutions and hence the problem \eqref{main} also has infinitely many solutions.
\subsection{Boundedness of any solution to \eqref{main}}
\noindent The idea carved out here is a usual one that appears in most literatures and hence we will only sketch that an improvement in the integrability is possible upto $L^{\infty}$ assuming an integrability of certain order, say $p$. The boundedness will follow from a {\it bootstrap} argument. Without loss of generality we consider the set $\Omega'=\{x\in\Omega^c:u(x)>1\}$ and thus from the positivity of a fixed solution, say, $u$ we have $u=u^+>0$ a.e. in $\Omega$. Let $u\in L^{\beta}(\Omega^c)$ for $\beta>1$. On testing with $u^{\beta}$ to obtain the following.	
\begin{align}\label{bddness}
\begin{split}
&\frac{1}{2}\langle u,u^{\beta}\rangle+\int_{\Omega^c}u^{\beta+1}dx\\
&=\left(\lambda\int_{\Omega'}|u|^{\beta-\gamma}dx+\int_{\Omega'}|u|^{p-1+\beta}dx\right)\frac{(\beta+1)^2}{4\beta}\\
&\leq\left(\lambda\int_{\Omega'}|u|^{\beta}(1+|u|^{p-1})dx\right)\frac{(\beta+1)^2}{4\beta};~\text{since in}~\Omega'~\text{we have}~u>1\\
&\leq\left(\lambda\int_{\Omega'}|u|^{\beta}|u|^{p}dx\right)\frac{(\beta+1)^2}{4\beta}\\
&\leq\lambda \beta C''\|u\|_{\alpha^*}^{p}\|u^{\beta}\|_t;~\text{since by using the H\"{o}lder's inequality}.
\end{split}
\end{align}
Here $t=\frac{\alpha^*}{\alpha^*-p}$ for some $\alpha^*>1$, $t^*=\frac{tN}{N-ts}<2_{s^-}^*$. We further have 
\begin{align}\label{bddness1}
\begin{split}
C'\|u^{\frac{\beta}{2}}\|_{\alpha^*}^2&\leq C'\|u^{\frac{\beta+1}{2}}\|_{\alpha^*}^2\\
&\leq \int_{\Omega'}\frac{\left|u(x)^{\frac{(\beta+1)}{2}}-u(y)^{\frac{(\beta+1)}{2}}\right|^2}{|x-y|^{N+2s}}dxdy.
\end{split}
\end{align}
So, from \eqref{bddness} and \eqref{bddness1}, the story so far is as follows.
\begin{align}
C'\|u^{\frac{\beta}{2}}\|_{\alpha^*}^2&\leq \lambda \beta C''\|u\|_{\alpha^*}^{p}\|u^p\|_t.
\end{align}
For the fixed $\alpha^*>1$ so chosen, we set $\eta=\frac{\alpha^*}{2t}>1$ for a suitable choice of $t$ and $\tau=t\beta$ to get
\begin{align}\label{moser1}
\|u\|_{\eta\tau}&\leq (\beta C)^{t/\tau}\|u\|_{\tau};~\text{where}~C=\lambda  C''\|u\|_{\alpha^*}^{p}~\text{is a fixed  quantity for a fixed solution}~u.\\
\end{align}
Let us now iterate with $\tau_0=t$, $\tau_{n+1}=\eta\tau_n=\eta^{n+1}t$. After $n$ iterations, the inequality \eqref{moser1} yields
\begin{align}\label{moser2}
\|u\|_{\tau_{n+1}}&\leq C^{\sum\limits_{i=0}^{n}\frac{t}{\tau_i}}\prod\limits_{i=0}^{n}\left(\frac{\tau_i}{t}\right)^{\frac{t}{\tau_i}}\|u\|_t.
\end{align}
By using the fact that $\eta>1$ and the iterative scheme, i.e. $\tau_0=t$, $\tau_{n+1}=\eta\tau_n=\eta^{n+1}t$,
we get $$\sum\limits_{i=0}^{\infty}\frac{t}{\tau_i}=\sum\limits_{i=0}^{\infty}\frac{1}{\eta^i}=\frac{\eta}{\eta-1}$$ and 
$$\prod\limits_{i=0}^{\infty}\left(\frac{\tau_i}{t}\right)^{\frac{t}{\tau_i}}=\eta^{\frac{\eta^2}{(\eta-1)^2}}.$$
Therefore, on passing the limit $n\rightarrow\infty$ in \eqref{moser2}, we obtain 
\begin{align}\label{moser3}
\|u\|_{\infty}&\leq C^{\frac{\eta}{\eta-1}}\eta^{\frac{\eta^2}{(\eta-1)^2}}\|u\|_t.
\end{align}
Thus $u\in L^{\infty}(\Omega^c)$.
\section{Appendix}
\noindent We now claim that
\begin{align}\label{conv sing}
\lim\limits_{n\rightarrow+\infty}\int_{\Omega^c}|u_n|^{-\gamma-1}u_n vdx=\int_{\Omega^c}|u|^{-\gamma-1}u vdx<\infty~\text{for}~v~\text{bounded support}.
\end{align}
{\bf{Proof of the claim:}}
\noindent Let us denote the set $A_n=\{x\in\Omega^c:u_n(x)=0\}$. Since $v u_n^{-\gamma}\in L^1(\Omega^c)$, we have that the Lebesgue measure of $A_n$ is zero, i.e. $|A_n|=0$. Thus by the sub-additivity of the Lebesgue measure we have, $|\bigcup A_n|=0$. Let $x\in\text{supp}(v)\setminus D$ such that $u(x)=0$. Here $|D|<\delta$ - obtained from the Egorov's theorem - where $u$ is a uniform limit of (a subsequence of $\{u_n\}$, still denoted as $\{u_n\}$) $u_n$ in $\text{supp}(v)\setminus D$. Further, define $$A_{m,n}=\{x\in\text{supp}(v)\setminus D:|u_n(x)|<\frac{1}{m}\}.$$
Note that due to the uniform convergence, for a fixed $n$ we have $|A_{m,n}|\rightarrow 0$ as $m\rightarrow\infty$. Now consider 
$$\underset{m,n\in\mathbb{N}}{\bigcup}A_{m,n}=\underset{n\geq 1}{\bigcup}\underset{m\geq n}{\bigcap}A_{m,n}.$$
Observe that, for a fixed $n$, $$|\underset{m\geq n}{\bigcap}A_{m,n}|=\underset{m\rightarrow\infty}{\lim}A_{m,n}=0.$$
The above argument is true for each fixed $n$ and thus $$|\underset{m,n\in\mathbb{N}}{\bigcup}A_{m,n}|=0.$$
Therefore, $|\{x\in\text{supp}(v)\setminus D: u_n(x)\rightarrow u(x)=0\}|=0$. Hence $v\neq 0$ a.e. in $\text{supp}(v)$. Repeating this process by considering different test functions with increasing size of supports in $\Omega^c$ gives the same outcome. Hence $w\neq 0$ a.e. $\Omega^c$.

\end{document}